\documentclass{amsart}
\usepackage{enumerate}
\usepackage{hyperref}
\usepackage{subfig,amssymb}
\usepackage{color,tikz}
\usetikzlibrary{arrows,positioning,shapes,fit,calc}
\usetikzlibrary{backgrounds,scopes}   %
\usepackage[utf8]{inputenc}
\usepackage[T1]{fontenc}
\usepackage{lmodern}

\newenvironment{subproof}[1][\proofname]{%
  \begin{proof}[#1]%
}{%
  \end{proof}%
}

\hypersetup{
    pdftitle=   {Strongly even-cycle decomposable graphs},
   pdfauthor=  {Tony Huynh, Andrew D. King, Sang-il Oum, Maryam Verdian-Rizi}
}
\newtheorem{THM}{Theorem}[section]
\newtheorem*{THM*}{Theorem~\ref{main}}
\newtheorem*{THMTH*}{Theorem~\ref{main-threshold}}
\newtheorem*{THMJOIN}{Theorem~\ref{thm:join}}
\newtheorem*{THMSUB}{Theorem~\ref{sub}}
\newtheorem*{THMTWIN}{Theorem~\ref{twin}}
\newtheorem*{THMTWIN2}{Theorem~\ref{thm:twin2}}
\newtheorem*{CORCO}{Corollary~\ref{cor:cographs}}
\newtheorem{LEM}[THM]{Lemma}
\newtheorem{COR}[THM]{Corollary}

\newtheorem{CLAIM}{Claim}

\theoremstyle{remark}

\newcommand\abs[1]{\lvert #1\rvert}

\theoremstyle{definition}
\newtheorem{DEFN}[THM]{Definition}
\begin{document}
\title{Strongly even-cycle decomposable graphs}
\author{Tony Huynh}
\author{Andrew D. King}
\author{Sang-il Oum}
\author{Maryam Verdian-Rizi}
\address[Tony Huynh]{Département de Mathématique, Université libre de Bruxelles, Boulevard du Triomphe, B-1050 Brussels, Belgium
}
\address[Sang-il Oum and Maryam Verdian-Rizi]{Department of Mathematical Sciences, KAIST, 291 Daehak-ro
  Yuseong-gu Daejeon, 34141 South Korea}
\address[Andrew D. King]{Department of Mathematics, Simon Fraser
  University, 8888 University Drive, Burnaby, BC, V5A 1S6, Canada}
\email{tony.bourbaki@gmail.com}
\email{andrew.d.king@gmail.com}
\email{sangil@kaist.edu}
\email{mverdian@gmail.com}
\thanks{T.~H., S.~O., and M.~V.-R. are supported by Basic Science Research
  Program through the National Research Foundation of Korea (NRF)
  funded by the Ministry of Science, ICT \& Future Planning
  (2011-0011653).  T.~H. is also supported by the European Research Council under the European Union's Seventh Framework Programme (FP7/2007-2013)/ERC Grant Agreement no. 279558. A.~D.~K. is supported by an EBCO/Ebbich Postdoctoral Scholarship and the NSERC Discovery Grants of Pavol Hell and Bojan Mohar.}
\date{\today}

\begin{abstract}
    A graph is \emph{strongly even-cycle decomposable} if the edge set of every subdivision with an even number of edges can be partitioned into cycles of even length. We prove that several fundamental composition operations that preserve the property of being Eulerian also yield strongly even-cycle decomposable graphs.  As an easy application of our  theorems, we give an exact characterization of the set of strongly even-cycle decomposable cographs.
\end{abstract}


\maketitle

\section{Introduction}\label{sec:intro}
A graph $G$ is \emph{even-cycle decomposable} if the edge set of $G$ can be partitioned into cycles of even length.  Even-cycle decomposable graphs have been the subject of substantial attention.  For a summary of relevant results we refer  the reader to the surveys of Jackson~\cite{jackson} or Fleischner~\cite{fleischner}.  
In addition, Huynh, Oum and Verdian-Rizi~\cite{oddk5} recently investigated even-cycle decomposable graphs with respect to odd minors.

In this paper we instead focus on a stronger decomposition property.
Namely, we define a graph $G$ to be \emph{strongly even-cycle decomposable} if every subdivision of $G$ with an even number of edges is even-cycle decomposable.  
Note that a strongly even-cycle decomposable graph $G$ with $\abs{V(G)}\ge 3$ and no isolated vertices is necessarily Eulerian, loopless and $2$-connected.
For us, an \emph{Eulerian graph} will always mean a (not necessarily connected) graph in which all vertex degrees are even. An \emph{anti-Eulerian graph} is a graph in which every vertex has odd degree.

One motivation for introducing strongly even-cycle decomposable graphs is that inductive arguments tend to work more smoothly for strongly even-cycle decomposable graphs as opposed to even-cycle decomposable graphs.  For this reason, sometimes the easiest way to prove that a graph is even-cycle decomposable is to prove that it is strongly even-cycle decomposable.  In addition, there are indeed interesting classes of graphs that are strongly even-cycle decomposable.  For example,
Seymour~\cite{planar} proved that planar graphs that satisfy the obvious necessary conditions are strongly even-cycle decomposable.

\begin{THM}[Seymour \cite{planar}] \label{planar}
Every loopless $2$-connected Eulerian planar graph is strongly even-cycle decomposable.
\end{THM}

Theorem~\ref{planar} was generalized by Zhang~\cite{nok5} to graphs having no $K_5$-minor.

\begin{THM}[Zhang \cite{nok5}] \label{noK5}
Every loopless $2$-connected Eulerian $K_5$-minor-free graph is strongly even-cycle decomposable.
\end{THM}

We note that $K_5$ is an example where the obvious necessary conditions are not sufficient.  To see this, observe that $K_5$ has $10$ edges but every even-length cycle of $K_5$ is a $4$-cycle.  Thus, $K_5$ is not even-cycle decomposable (and hence not strongly even-cycle decomposable).  Markstr\"om~\cite{markstrom} recently gave a construction for an infinite class of $4$-regular $2$-connected graphs that are not even-cycle decomposable; the construction is based on a gadget that places $K_5$ in an edge.

In this paper we prove that several fundamental composition operations that preserve the property of being Eulerian also yield strongly even-cycle decomposable graphs.  
Our main composition operations (from which the others can be derived) are the following.  Definitions are deferred until later.

\begin{THMSUB}
    Let $G$ be a simple strongly even-cycle decomposable graph, let $v$ be a non-isolated vertex of $G$, and let $H$ be a simple Eulerian graph with an odd number of vertices.  Then the substitution of $v$ by $H$ in $G$ is strongly even-cycle decomposable, provided that  $H$ is not $K_3$ or $\deg_G(v)\ge 4$.
\end{THMSUB}
\begin{figure}[t]
  \centering
  \tikzstyle{every node}=[circle,draw,fill=black!50,inner sep=0pt,minimum width=4pt]
  \begin{tikzpicture}[scale=.75]
    \draw (90:1) node {} -- (210:1) node {} -- (330:1) node {}
    --cycle;
    \draw (0,0) node{};
  \end{tikzpicture}
  \caption{A co-claw.}
  \label{fig:coclaw}
\end{figure}
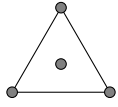

\begin{THMTWIN}
    Let $G$ be a simple strongly even-cycle decomposable graph and let $u$ and $v$ be non-isolated, non-adjacent twin vertices of $G$. If $H$ is a simple Eulerian graph with an even number of vertices, then the twin substitution of $\{u,v\}$ by $H$ in $G$ is also strongly even-cycle decomposable, provided that $H$ is not a co-claw or $\deg_G(v)\ge 4$.
\end{THMTWIN}

\begin{THMTWIN2}
    Let $G$ be a simple strongly even-cycle decomposable graph and let $u$ and $v$ be adjacent twin vertices of $G$ having degree at least $4$. If $H$ is a simple anti-Eulerian graph, then the twin substitution of $\{u,v\}$ by $H$ in $G$ is also strongly even-cycle decomposable.
\end{THMTWIN2}

\begin{THMJOIN}
    Let $G$ be a simple Eulerian graph that is a join of two graphs each having at least two vertices. Then $G$ is strongly even-cycle decomposable if and only if $G$ is neither $K_5$ nor $K_5$ with an edge subdivided.
\end{THMJOIN}
See Figure~\ref{fig:coclaw} for a co-claw and Figure~\ref{fig:k5sub} for $K_5$ with an edge subdivided.
These theorems suggest that $K_5$ is essentially the unique obstruction to the strongly even-cycle decomposable property.  In particular, we obtain the following corollary as an easy application of our 
composition theorems. A \emph{cograph} is a simple graph with no induced path of length $3$. Here is an exact characterization of the set of strongly even-cycle
decomposable cographs.

\begin{figure}
\centering
   \begin{tikzpicture}[scale=.75]
    \tikzstyle{v}=[circle,draw,fill=black!50,inner sep=0pt,minimum width=4pt]
    \foreach \x in {1,2,3,4,5} {
        \node[v] at (360*\x/5+90:1) (v\x) {};
    }
    \foreach \x in {1,2,3,4} {
        \foreach \y in {\x,...,5} {
            \draw (v\x)--(v\y);
    }   
    }
    \node [v] at ($(v2)!.5!(v3)$)   {};
    \end{tikzpicture}
\caption{$K_5$ with an edge subdivided.}
\label{fig:k5sub}
\end{figure}
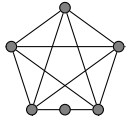
\begin{CORCO}
    Let $G$ be a cograph with no isolated vertices. Then $G$ is strongly even-cycle decomposable if and only if $G$ is $2$-connected Eulerian and $G$ is neither $K_5$ nor $K_5$ with an edge subdivided.
\end{CORCO}

The rest of the paper is organized as follows.  In Section~\ref{sec:sign}, we define \emph{signed graphs} and recast our problem in their language.  In Section~\ref{sec:small} we prove that Eulerian complete bipartite graphs are strongly even-cycle decomposable.  In Section~\ref{sec:smallgraphs}, we present several classes of strongly even-cycle decomposable graphs to be used in other sections as base cases. In Sections~\ref{sec:sub},~\ref{sec:twinsub}, and~\ref{sec:join} we introduce our composition operations and prove that they yield strongly even-cycle decomposable graphs.  
Finally, in Section~\ref{sec:cographs} we give a quick derivation of Corollary~\ref{cor:cographs}.

\section{Signed graphs}\label{sec:sign}
A \emph{signed graph} is a pair $(G, \Sigma)$ consisting of a graph $G$ together with a \emph{signature} $\Sigma \subseteq E(G)$.  The edges in $\Sigma$ are \emph{odd} and the other edges are \emph{even}.  We extend this terminology and define a cycle (or path) to be \emph{even} or \emph{odd} according as it contains an even or odd number of odd edges.  To avoid confusion, we will always use the term \emph{even-length cycle} if we need to refer to a cycle with an even number of edges. 
For a graph $G$, a subset of $E(G)$ is called a \emph{signing} of $E(G)$; a signing $\Sigma$ of $E(G)$ induces a signed graph $(G,\Sigma)$.
 
A signed graph $(G, \Sigma)$ is \emph{even-cycle decomposable}, if $E(G)$ can be partitioned into even cycles. 
The relationship between even-cycle decompositions of signed graphs and graphs is as follows.  Let $G$ be a graph and let $H$ be a subdivision of $G$ with $\abs{E(H)}$ even.  Consider the signed graph $\mathcal{G}_H:=(G, \Sigma)$, where $e \in \Sigma$ if and only if the subdivided path in $H$ corresponding to $e$ has an odd number of edges.  Observe that $H$ is even-cycle decomposable if and only if $\mathcal{G}_H$ is even-cycle decomposable.  Thus, we have the following equivalent definition of strongly even-cycle decomposable graphs.

\begin{DEFN}
A graph $G$ is strongly even-cycle decomposable if and only if for each signing $\Sigma$ of $E(G)$ with $\abs{\Sigma}$ even, the signed graph $(G, \Sigma)$ is even-cycle decomposable.
\end{DEFN}

For $X \subseteq V(G)$, we let $\delta_G(X)$ be the set of non-loop edges with exactly one end in $X$.  We say that $\delta_G(X)$ is the \emph{cut induced by $X$}.  Two signatures $\Sigma_1, \Sigma_2 \subseteq E(G)$ are \emph{equivalent} if their symmetric difference is a cut.  Note that signature equivalence is an equivalence relation.  The operation of changing to an equivalent signature is called \emph{re-signing}.  The main reason for working with signed graphs is that if $\Sigma_1, \Sigma_2 \subseteq E(G)$ are equivalent signatures, then $(G, \Sigma_1)$ and $(G, \Sigma_2)$ have exactly the same set of even cycles.
Thus, for equivalent signatures $\Sigma_1$ and $\Sigma_2$, $(G, \Sigma_1)$ is even-cycle decomposable if and only if $(G, \Sigma_2)$ is even-cycle decomposable.   
The \emph{parity} of a vertex $v$ in a signed graph $(G, \Sigma)$ is the parity of
the number of odd non-loop edges incident with $v$.  

We will frequently use the following well-known lemma without explicit reference.

\begin{LEM} \label{forest}
Let $(G, \Sigma)$ be a signed graph.  For any $F \subseteq E(G)$ which does not contain a cycle, there is a signature $\Sigma_F$ which is equivalent to $\Sigma$ such that $\Sigma_F \cap F = \emptyset$.
\end{LEM}

Observe that Lemma~\ref{forest} implies that $(G, \Sigma_1)$ and $(G, \Sigma_2)$ 
have the same set of even cycles if and only if $\Sigma_1$ and $\Sigma_2$ are equivalent.

An \emph{almost even-cycle decomposition} of a signed graph $(G, \Sigma)$ is a partition
of $E(G)$ into cycles, at most one of which is odd.
\begin{LEM}\label{almost}
  Let $G$ be a strongly even-cycle decomposable graph and let $\Sigma$ be a signing of $E(G)$. 
  For each edge $e$ of $G$, there exists an almost even-cycle decomposition $\{C\}\cup \mathcal C$ of $(G,\Sigma)$ such that $e\in E(C)$ and all cycles in $\mathcal C$ are even.
\end{LEM}
\begin{proof}
    We may assume that $\abs{\Sigma}$ is odd.
  Let $\Sigma'=\Sigma\cup\{e\}$ if $e\notin \Sigma$ and $\Sigma'=\Sigma\setminus\{e\}$ otherwise.
  Since 
  $G$ is strongly even-cycle decomposable, 
  $(G,\Sigma')$ has an even-cycle decomposition. Such an 
  even-cycle decomposition corresponds to an almost even-cycle decomposition of $(G,\Sigma)$ whose odd cycle contains $e$.
\end{proof}

\section{Complete bipartite graphs}\label{sec:small}
In this section, we prove that Eulerian complete bipartite graphs are strongly even-cycle decomposable.  In fact, we will need to prove something slightly stronger.  Namely, after removing the edges of a $4$-cycle from an Eulerian complete bipartite graph, the resulting graph is strongly even-cycle decomposable.

Let us write $K_{n,m}-C_4$ to denote a
subgraph of $K_{n,m}$ obtained by deleting the edges of a fixed $4$-cycle
of $K_{n,m}$. We proceed via a sequence of lemmas.

\begin{LEM} \label{basebasecase}
$K_{2,n}$ is strongly even-cycle decomposable for all positive even integers $n$.
\end{LEM}
\begin{proof}
Consider a signing $\Sigma$ of $E(K_{2,n})$ with $\abs{\Sigma}$ even.  Let $\{u,v\}$ and $\{x_1, \dots, x_n\}$ be the bipartition of $K_{2,n}$.  By Lemma~\ref{forest}, we may assume that all edges incident with $u$ are even.  By re-indexing vertices if necessary, we may assume that there exists an even index $i$ such that the edges $vx_1, \dots, vx_i$ are all odd and all the other edges incident with $v$ are even.  But now,
\[
\{ux_1vx_2, ux_3vx_4, ux_5vx_6, \dots, ux_{n-1}vx_n\}
\]
is an even-cycle decomposition of $(K_{2,n}, \Sigma)$.
\end{proof}

\begin{LEM} \label{basecase}
$K_{4,4}-C_4$ is strongly even-cycle decomposable.
\end{LEM}
\begin{proof}
Let $G:=K_{4,4}-C_4$ with bipartition $\{a,b,c,d\}$ and $\{w,x,y,z\}$ such that $C_1:=aybz$, $C_2:=cydz$, and $C_3:=wcxd$ are edge-disjoint $4$-cycles and the deleted $C_4$ is $awbx$. Let $\Sigma$ be a signing of $E(G)$ with $\abs{\Sigma}$ even.  If either $C_1$ or $C_3$ is even then $(G, \Sigma)$ is even-cycle decomposable by Lemma~\ref{basebasecase}.  So we may assume $C_1$ and $C_3$ are both odd and therefore $C_2$ is even.

\begin{figure}
\begin{center}
\begin{tikzpicture}[scale=.85]
  \tikzstyle{vertex}=[circle,fill=black!25,minimum size=12pt,inner sep=0pt]

  \foreach \name/\x in {a/0, y/1, c/2, w/3}
    \node[vertex] (G-\name) at (\x,0) {$\name$};

  \foreach \name/\x in {b/0, z/1, d/2, x/3}
    \node[vertex] (G-\name) at (\x,-1) {$\name$};

  \foreach \from/\to in {a/y,b/y,a/z,b/z}
    \draw[dashed,thick] (G-\from) -- (G-\to);
  \foreach \from/\to in {c/y,d/y,c/z,d/z}
    \draw[dashed,thick] (G-\from) -- (G-\to);
  \foreach \from/\to in {c/w,d/w,c/x,d/x}
    \draw[dashed,thick] (G-\from) -- (G-\to);
  \draw[thick] (G-a) -- (G-y);
  \draw[thick] (G-c) -- (G-w);

  \draw[->] (4,-.5) -- (5,-.5);

  \foreach \name/\x in {a/0, y/1, c/2, w/3}
    \node[vertex] (G-a\name) at (\x+6,0) {$\name$};
  \foreach \name/\x in {b/0, z/1, d/2, x/3}
    \node[vertex] (G-a\name) at (\x+6,-1) {$\name$};

  \foreach \name/\x in {a/0, y/1, c/2, w/3}
    \node[vertex] (G-b\name) at (\x+10.5,0) {$\name$};

  \foreach \name/\x in {b/0, z/1, d/2, x/3}
    \node[vertex] (G-b\name) at (\x+10.5,-1) {$\name$};
  \foreach \from/\to in {aa/ay,aa/az,az/ac,ay/ad,aw/ac,aw/ad}
    \draw[dashed,thick] (G-\from) -- (G-\to);
  \foreach \from/\to in {aa/ay,ac/aw}
    \draw[thick] (G-\from) -- (G-\to);
  \foreach \from/\to in {bb/by,bb/bz,bz/bd,by/bc,bc/bx,bd/bx}
    \draw[dashed,thick] (G-\from) -- (G-\to);

\end{tikzpicture}\vspace{.5cm}

\begin{tikzpicture}[scale=.85]
  \tikzstyle{vertex}=[circle,fill=black!25,minimum size=12pt,inner sep=0pt]

  \foreach \name/\x in {a/0, y/1, c/2, w/3}
    \node[vertex] (G-\name) at (\x,0) {$\name$};

  \foreach \name/\x in {b/0, z/1, d/2, x/3}
    \node[vertex] (G-\name) at (\x,-1) {$\name$};

  \foreach \from/\to in {a/y,b/y,a/z,b/z}
    \draw[dashed,thick] (G-\from) -- (G-\to);
  \foreach \from/\to in {c/y,d/y,c/z,d/z}
    \draw[dashed,thick] (G-\from) -- (G-\to);
  \foreach \from/\to in {c/w,d/w,c/x,d/x}
    \draw[dashed,thick] (G-\from) -- (G-\to);
  \draw[thick] (G-a) -- (G-y);
  \draw[thick] (G-c) -- (G-w);
  \draw[thick] (G-y) -- (G-c);
  \draw[thick] (G-z) -- (G-d);

  \draw[->] (4,-.5) -- (5,-.5);

  \foreach \name/\x in {a/0, y/1, c/2, w/3}
    \node[vertex] (G-a\name) at (\x+6,0) {$\name$};
  \foreach \name/\x in {b/0, z/1, d/2, x/3}
    \node[vertex] (G-a\name) at (\x+6,-1) {$\name$};

  \foreach \name/\x in {a/0, y/1, c/2, w/3}
    \node[vertex] (G-b\name) at (\x+10.5,0) {$\name$};

  \foreach \name/\x in {b/0, z/1, d/2, x/3}
    \node[vertex] (G-b\name) at (\x+10.5,-1) {$\name$};
  \foreach \from/\to in {aa/ay,aa/az,az/ac,ay/ad,aw/ac,aw/ad}
    \draw[dashed,thick] (G-\from) -- (G-\to);
  \foreach \from/\to in {aa/ay,ac/aw}
    \draw[thick] (G-\from) -- (G-\to);
  \foreach \from/\to in {bb/by,bb/bz,bz/bd,by/bc,bc/bx,bd/bx}
    \draw[dashed,thick] (G-\from) -- (G-\to);
  \foreach \from/\to in {by/bc,bz/bd}
    \draw[thick] (G-\from) -- (G-\to);
\end{tikzpicture}\vspace{.5cm}

\begin{tikzpicture}[scale=.85]
  \tikzstyle{vertex}=[circle,fill=black!25,minimum size=12pt,inner sep=0pt]

  \foreach \name/\x in {a/0, y/1, c/2, w/3}
    \node[vertex] (G-\name) at (\x,0) {$\name$};

  \foreach \name/\x in {b/0, z/1, d/2, x/3}
    \node[vertex] (G-\name) at (\x,-1) {$\name$};

  \foreach \from/\to in {a/y,b/y,a/z,b/z}
    \draw[dashed,thick] (G-\from) -- (G-\to);
  \foreach \from/\to in {c/y,d/y,c/z,d/z}
    \draw[dashed,thick] (G-\from) -- (G-\to);
  \foreach \from/\to in {c/w,d/w,c/x,d/x}
    \draw[dashed,thick] (G-\from) -- (G-\to);
  \draw[thick] (G-a) -- (G-y);
  \draw[thick] (G-c) -- (G-w);
  \draw[thick] (G-y) -- (G-c);
  \draw[thick] (G-y) -- (G-d);

  \draw[->] (4,-.5) -- (5,-.5);

  \foreach \name/\x in {a/0, y/1, c/2, w/3}
    \node[vertex] (G-a\name) at (\x+6,0) {$\name$};
  \foreach \name/\x in {b/0, z/1, d/2, x/3}
    \node[vertex] (G-a\name) at (\x+6,-1) {$\name$};

  \foreach \name/\x in {a/0, y/1, c/2, w/3}
    \node[vertex] (G-b\name) at (\x+10.5,0) {$\name$};

  \foreach \name/\x in {b/0, z/1, d/2, x/3}
    \node[vertex] (G-b\name) at (\x+10.5,-1) {$\name$};
  \foreach \from/\to in {aa/ay,aa/az,az/ad,ad/ax,ax/ac,ac/ay}
    \draw[dashed,thick] (G-\from) -- (G-\to);
  \foreach \from/\to in {aa/ay,ay/ac}
    \draw[thick] (G-\from) -- (G-\to);
  \foreach \from/\to in {by/bb,bb/bz,bz/bc,bc/bw,bw/bd,bd/by}
    \draw[dashed,thick] (G-\from) -- (G-\to);
  \foreach \from/\to in {by/bd,bc/bw}
    \draw[thick] (G-\from) -- (G-\to);

\end{tikzpicture}\vspace{-.2cm}
\end{center}
\caption{Three possible signatures for $K_{4,4}-C_4$ in the proof of Lemma~\ref{basecase}, and corresponding even-cycle decompositions.  Odd edges are solid and even edges are dashed.}\label{fig:k44c4}
\end{figure}
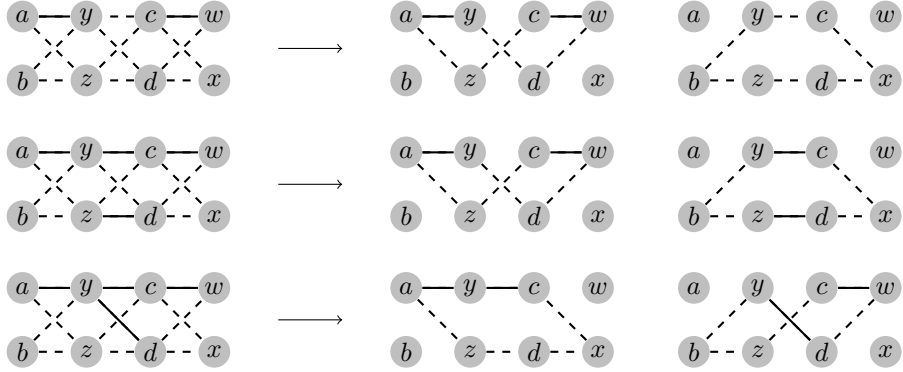

Note that $E(C_1)$, $E(C_2)$, and $E(C_3)$ are disjoint edge cuts.  We
may therefore assume by re-signing that each of $C_1$ and $C_3$
contains exactly one odd edge and that $C_2$ has either zero or two
odd edges.  Assume without loss of generality that the edges $ay$ and
$cw$ are odd.  There are three cases to consider: $C_2$ has no odd
edges, $C_2$ has two odd edges forming a matching (which we can
choose, by re-signing), or $C_2$ has two odd edges not forming a
matching (which can be chosen to meet at $y$ or $c$ by re-signing).
The even-cycle decompositions for each of these cases are shown in Figure~\ref{fig:k44c4}.
\end{proof}

\begin{LEM}\label{Km,nC4}
$K_{n,m}-C_4$ is strongly even-cycle decomposable for all even integers $n,m \geq 2$.
\end{LEM}

\begin{proof}
We proceed by induction on $n+m$.  Note that $K_{2,j}-C_4$ is $K_{2, j-2}$ (together with two isolated vertices). By Lemma~\ref{basebasecase} we may assume that $n,m \geq 4$.  Next, by Lemma~\ref{basecase}, we may assume that $m \geq 6$.  Now let $X:=\{x_1, \dots, x_n\}$ and $\{y_1, \dots, y_m\}$ be a bipartition of $K_{n,m}-C_4$, and let $\Sigma$ be a signing of its edges with $\abs{\Sigma}$ even.  By re-indexing, we may assume that the vertices of the missing $C_4$ are $\{x_1,x_2,y_1,y_2\}$.    
Since $m \geq 6$, two vertices in $\{y_3,y_4,\ldots,y_{m}\}$ (say, $y_{m-1}$ and $y_m$) must have the same parity.  Thus, the subgraphs induced by $X \cup \{y_{m-1},y_m\}$ and $X \cup \{y_1, \dots, y_{m-2}\}$ both contain an even number of odd edges.   But now we are finished since the first subgraph is isomorphic to $K_{2,n}$ (and hence is strongly even-cycle decomposable by Lemma~\ref{basebasecase}), while the second is isomorphic to $K_{n, m-2}-C_4$ (and is strongly even-cycle decomposable by the induction hypothesis).
\end{proof}
\begin{LEM}\label{Km,nsign}
  Let $n,m\ge 2$ be even integers.
  For every signing $\Sigma$ of $E(K_{n,m})$, there exists a $4$-cycle
  $C$ in $K_{n,m}$ such that $\abs{\Sigma\cap E(C)}\equiv\abs{ \Sigma}\pmod2$.
\end{LEM}
\begin{proof}
  Let $(A,B)$ be the bipartition of $K_{n,m}$.  By re-signing, we
  may assume that there exists a vertex $a_1 \in A$ such that all
  edges incident with $a_1$ are even.  Now, since $\abs{A}$ and the
  parity of $a_1$ are even, it follows that there is a vertex $a_2 \in
  A$ such that $a_2$ has the same parity as $\abs{\Sigma}$ and $a_2 \neq
  a_1$.   
  Therefore,  there exist $b_1,b_2\in B$ such that $C:=a_1b_1a_2b_2$
  is a $4$-cycle with the same parity as $\abs{\Sigma}$. 
\end{proof}
\begin{LEM}\label{Km,n}
  For all even  integers $n,m \geq 2$,
  $K_{n,m}$ is strongly even-cycle decomposable.
\end{LEM}

\begin{proof}
  Let $\Sigma$ be a signing of $E(K_{n,m})$ with $\abs{\Sigma}$ even.
  By Lemma~\ref{Km,nsign}, there exists an even $4$-cycle $C$ of $K_{n,m}$.
  We are done by Lemma~\ref{Km,nC4}.
\end{proof}

\section{Small graphs}\label{sec:smallgraphs}
We have previously noted that $K_5$ is not strongly even-cycle decomposable.  On the other hand, up to re-signing, there is only one even-size signature $\Sigma$ of $E(K_5)$ such that $(K_5, \Sigma)$ is not even-cycle decomposable.

\begin{LEM} \label{K5}
Let $\Sigma$ be a signing of $E(K_5)$ with $\abs{\Sigma}$ even. If $(K_5,\Sigma)$ is not even-cycle decomposable, then $\Sigma$ and $E(K_5)$ are equivalent.
\end{LEM}

\begin{proof}
Since the complement of a $5$-cycle in $K_5$ is another $5$-cycle, every $5$-cycle of $(K_5, \Sigma)$ must be odd, else $(K_5, \Sigma)$ is even-cycle decomposable.
Every $4$-cycle in $(K_5, \Sigma)$ is the symmetric difference of two $5$-cycles, and thus is even.
Every $3$-cycle in $(K_5, \Sigma)$  is the symmetric difference of a $4$-cycle and a $5$-cycle, and thus is odd.
Thus, $(K_5, \Sigma)$ and $(K_5, E(K_5))$ have the same set of even cycles, and so $\Sigma$ and $E(K_5)$ are equivalent. 
\end{proof}

For a graph $G$, $e \in E(G)$ and a non-negative integer $m$, we define $G+me$ to be the graph obtained from $G$ by adding $m$ additional edges in parallel with $e$.

\begin{LEM}\label{parallel}
  If $G$ is a strongly even-cycle decomposable graph, then $G+me$ is strongly even-cycle decomposable for all $e \in E(G)$ and all even $m$.  
\end{LEM}
\begin{proof}
 Consider a signing $\Sigma$ of $E(G+me)$ with $\abs{\Sigma}$ even. Let $t$ and $w$ be the ends of $e$.  Note that there is a set of $m/2$ edge-disjoint even $2$-cycles between
 $t$ and $w$ in $(G+me, \Sigma)$.  Thus, we can remove this set of even $2$-cycles and then use the fact that $G$ is strongly even-cycle decomposable.   
\end{proof}

The next lemma shows that the converse of Lemma~\ref{parallel} fails for $G=K_5$.

\begin{LEM} \label{K5+}
$K_5+me$ is strongly even-cycle decomposable for all even $m \geq 2$ and $e\in E(K_5)$.  
\end{LEM}
\begin{proof}
Let $e=tw$ and $\Sigma$ be a signing of $E(K_5+me)$ with $\abs{\Sigma}$ even.
By repeatedly removing even $2$-cycles between $t$ and $w$ in
$(K_5+me,\Sigma)$, we may assume $m=2$.  Let $e_1$ and $e_2$ be
edges between $t$ and $w$ with the same sign.  If $((K_5+2e) \setminus
e_1\setminus e_2, \Sigma \setminus \{e_1,e_2\})$ is not equivalent to 
$((K_5+2e) \setminus e_1\setminus e_2, E((K_5+2e)\setminus
e_1\setminus e_2))$, then we are done by Lemma~\ref{K5}.  
Thus, by re-signing, we may assume that all edges in $E(K_5+2e) \setminus
\{e_1, e_2\}$ are odd.  Either $e_1$ and $e_2$ are both odd, or $e_1$ and $e_2$ are both even.  Even-cycle decompositions for both possibilities are shown in Figure~\ref{fig:K5+}.
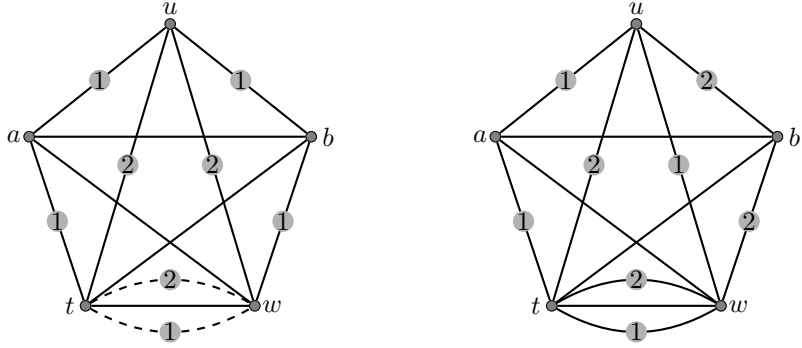
\begin{figure}
\centering
 \begin{tikzpicture}[scale=.75]
\tikzstyle{every node}=[circle,fill=black!30,inner sep=0pt,minimum width=4pt]
\node at (-1,-2) (v1) [label=left:$t$,draw,fill=black!50]{};
\node at (-2,1) (v2) [label=left:$a$,draw,fill=black!50]{};
\node at (.5,3) (v3) [label=above:$u$,draw,fill=black!50]{};
\node at (3,1) (v4) [label=right:$b$,draw,fill=black!50]{};
\node at (2,-2) (v5) [label=right: $w$,draw,fill=black!50] {};
\draw [thick] (v2) --node[midway] {1} (v3) -- node[midway] {1}(v4);
\draw[dashed,bend right,thick] (v1) to node[midway] {1} (v5);
\draw[dashed,bend left,thick] (v1) to node[midway] {2} (v5);
\draw[thick](v2) -- (v4);
\draw[thick](v1) -- (v5);
\draw[thick] (v2) --node[midway] {1}(v1) --node[midway] {2} (v3) -- node[midway] {2}(v5) --node[midway] {1} (v4);
\draw[thick] (v2) -- (v5);
\draw[thick](v4) -- (v1);
\end{tikzpicture}\hspace{5em}%
\begin{tikzpicture}[scale=.75]
\tikzstyle{every node}=[circle,fill=black!30,inner sep=0pt,minimum width=4pt]
\node at (-1,-2) (v1) [label=left:$t$,draw,fill=black!50]{};
\node at (-2,1) (v2) [label=left:$a$,draw,fill=black!50]{};
\node at (.5,3) (v3) [label=above:$u$,draw,fill=black!50]{};
\node at (3,1) (v4) [label=right:$b$,draw,fill=black!50]{};
\node at (2,-2) (v5) [label=right: $w$,draw,fill=black!50] {};
\draw [thick] (v2) --node[midway] {1} (v3) -- node[midway] {2}(v4);
\draw[bend right,thick] (v1) to node[midway] {1} (v5);
\draw[bend left,thick] (v1) to node[midway] {2} (v5);
\draw[thick](v2) -- (v4);
\draw[thick](v1) -- (v5);
\draw[thick] (v2) --node[midway] {1}(v1) --node[midway] {2} (v3) -- node[midway] {1}(v5) --node[midway] {2} (v4);
\draw[thick] (v2) -- (v5);
\draw[thick](v4) -- (v1);
\end{tikzpicture}
\caption{Even-cycle decompositions of signed graphs in the proof of
  Lemma~\ref{K5+}. Solid edges are odd and dashed edges are even. The bent edges represent $e_1$ and $e_2$.
 The labelled edges illustrate the union of two even cycles  whose removal leaves an even $4$-cycle.}
\label{fig:K5+}
\end{figure}
\end{proof}

Let $\overline{K_n}$ denote the graph with $n$ vertices and no edges. 
Let $G$ and $H$ be graphs.  The \emph{join of $G$ and $H$} is the
graph obtained from the disjoint union of $G$ and $H$ by 
adding an edge $uv$ for all $u\in V(G)$ and $v\in V(H)$.
Two distinct vertices $u$ and $v$ of a graph $G$ are \emph{twins} if
no vertex in $V(G)\setminus\{u,v\}$ is adjacent to exactly one of $u$ and $v$.

\begin{LEM} \label{lem:parityreduction}
    Let $G$ be a graph having pairwise non-adjacent twins $x$, $y$, $z$ such that no loops or parallel edges are incident with $x$, $y$ or $z$. If $G-x-y$ is strongly even-cycle decomposable, then $G$ is strongly even-cycle decomposble. 
\end{LEM}

\begin{proof}
    Let $\Sigma$ be a signing of $E(G)$ with $\abs{\Sigma}$ even. By the pigeonhole principle, there are distinct vertices $v_1,v_2\in \{x,y,z\}$ such that the parity of $v_1$ is equal to that of $v_2$. Note that the subgraph of $G$ induced by all edges between $G-x-y-z$ and $\{v_1,v_2\}$ is strongly even-cycle decomposable by Lemma~\ref{basebasecase} and $G-v_1-v_2$ is isomorphic to $G-x-y$, which is strongly even-cycle decomposable.
\end{proof}

\begin{LEM} \label{apex}
Let $G$ be a simple Eulerian graph with at most one cycle with $\abs{V(G)}$ even.  Let $n$ be a positive integer. Then the join of $G$ with $\overline{K_{2n}}$ is  strongly even-cycle decomposable if and only if $n>1$ or $G$ is not a co-claw.
\end{LEM}

\begin{proof}
    The forward implication is trivial because the join of a co-claw and $\overline{K_2}$ is a subdivision of $K_5$.

For the other direction, we proceed by induction on $\abs{E(G)}+n$.
Let $G'$ be the join of $G$ and $\overline{K_{2n}}$ and let $\Sigma$ be a signing
of $E(G')$ with $\abs{\Sigma}$ even.  

If $\abs{\Sigma\cap E(G)}$ is even, then we are done by
Lemma~\ref{Km,n}. Thus, we may assume
$\abs{\Sigma\cap E(G)}$ is odd, and so $E(G)=E(C)$, where $C$ is an odd cycle.     

Let $X=V(\overline{K_{2n}})$ and  $K_{C, X}$ be the subgraph of $G'$ consisting of all edges between $V(C)$ and $X$.  
Suppose that $K_{C, X}$ contains an odd $4$-cycle $D$. Observe that
$C\cup D$ can be decomposed into two even cycles. The remaining edges are even-cycle decomposable by Lemma~\ref{Km,nC4}. Thus, we may assume $K_{C, X}$ does not contain an odd $4$-cycle.  
Since the $4$-cycles of $K_{C,X}$ span its cycle space, all cycles of $K_{C,X}$ are even.  
By Lemma~\ref{forest}, we may assume that all edges of $K_{C,X}$ are even.  Since $\abs{\Sigma}$ is even, but $C$ is odd, there must be a vertex $u \in V(G) \setminus V(C)$ of odd parity in $(G',\Sigma)$.  
We handle two separate cases depending on the size of $X$.  

\medskip
\noindent\textbf{Case 1.} $\abs{X}=2$.  In this case, by hypothesis, we may assume that $G$ is not a co-claw. Let $a$ and $b$ be the two vertices in $X$.  Since $u$ has odd parity, the edges $ua$ and $ub$ have different signs.  We first suppose that $C$ has length at least 4.  We choose two vertices $x$, $y$ in $C$ such that one of the $x$-$y$ paths in $C$, say $P_1$, is even with
at most two edges. Since $C$ has length at least $4$, the other $x$-$y$ path $P_2$ has an internal vertex $z$.  But then
$D_1=\{xb,bz,za,ay\}\cup E(P_1)$ and $D_2=\{xa,au,ub,by\}\cup E(P_2)$ are even cycles. By Lemma~\ref{Km,n},
the remaining edges are even-cycle decomposable.

We may therefore assume that $C$ is a triangle.  Since $G$ is not a
co-claw and $\abs{V(G)}$ is even, 
the number $m$ of isolated vertices in $G$ is at least $3$ and odd.
But now $G'$ is a subdivision of $K_5+(m-1)e$, and so we are done
by Lemma~\ref{K5+}.    

\medskip\noindent
\textbf{Case 2.} $\abs{X}\geq 4$. 
In this case, we can apply Lemma~\ref{lem:parityreduction} and the induction hypothesis, unless $\abs{X}= 4$
and $G$ is a co-claw.  Let $X=\{a,b,c,d\}$.
Again recall that $C$ is odd, and we have re-signed so that all edges between $C$ and $X$ are even.
Let $xy$ be an odd edge in $C$. Since one of $P_1:=aub$ and $P_2:=cud$ is odd and the other is even,
$C\cup P_1\cup \{xa,yb,xc,yd\}\cup P_2$ can be partitioned into two
even cycles. We are done since the signed graph of
remaining edges of $G'$ is Eulerian with no odd edge, and hence is even-cycle decomposable.
\end{proof}

\begin{LEM} \label{cliquejoin}
   Let $G$ be a simple Eulerian graph with at most one cycle with $\abs{V(G)}$ odd.  Then the join of $G$ and $K_2$ is strongly even-cycle decomposable if and only if $G\neq K_3$.
\end{LEM}
\begin{proof}
The forward implication is trivial because the join of $K_3$ and $K_2$ is $K_5$.

Let $G'$ be the join of $G$ and $K_2$. Let $G''$ be the join of $G\cup \{u\}$ and the $\overline{K_{2}}$,
where $u$ is a new isolated vertex added to $G$. Since $G''$ is a subdivision of $G'$, it follows that $G'$ is strongly even-cycle decomposable if and
only if $G''$ is strongly even-cycle decomposable. Thus, provided $G \neq K_3$, by Lemma~\ref{apex},
$G'$ is strongly even-cycle decomposable.
\end{proof}

\begin{LEM} \label{2-cyclejoin}
Let $G_1$ be the disjoint union of one or two $2$-cycles together with an even number of isolated vertices. Then, the join of $G_1$ and $\overline{K_{2n}}$ is strongly even-cycle decomposable for all $n \geq 1$.  
\end{LEM}

\begin{proof}
Let $G_2=\overline{K_{2n}}$, $G$ be the join of $G_1$ and $G_2$, and $\Sigma$ be a signing of $E(G)$ with $\abs{\Sigma}$ even. By Lemma~\ref{lem:parityreduction}, we may assume that $\abs{V(G_2)}=2$ and that $G_1$ has $0$ or $2$ isolated vertices.

\begin{figure}
\centering
    \begin{tikzpicture}[scale=.75]
      \tikzstyle{v}=[circle,draw,fill=black!50,inner sep=0pt,minimum width=4pt]
      \node[v] at (0,0) (w0) {};
      \node[v] at (0,1) (w1){};
      \node[v] at (0,-1)(w2){};
      \node[v] at (0,-2)(w3){};
      \node[v] at (1,.5)(v1){};
      \node[v] at (1,-.5)(v2){};
      \draw [dashed] (v1)--(w0)[bend right=15] to (w1)--(v2)--(w2)--(v1);
      \draw (v1)--(w1)[bend right=15] to (w0)--(v2)--(w3)--(v1);
    \end{tikzpicture}
\caption{A decomposition of a graph into two $5$-cycles in the proof of Lemma~\ref{2-cyclejoin}.}
\label{fig:2cyclejoin1}
\end{figure}
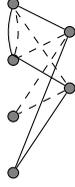

We first handle the case that $G_1$ has only one $2$-cycle $C$.  We may assume that $C$ is odd, else we can remove $E(C)$ and apply Lemma~\ref{Km,n}.  If $G_1$ has no isolated vertices, then $E(G)$ can be decomposed into two even triangles.  So, we may assume that $G_1$ has exactly two isolated vertices.  Figure~\ref{fig:2cyclejoin1} shows a decomposition of $E(G)$ into two $5$-cycles. Since the $2$-cycle $C$ is odd, we can ensure that both these $5$-cycles are even, as required. 

\begin{figure}
\centering
    \begin{tikzpicture}[scale=.75]
      \tikzstyle{v}=[circle,draw,fill=black!50,inner sep=0pt,minimum width=4pt]
      \node[v] at (0,0) (w0) {};
      \node [v] at (0,1) (w1){};
      \node[v] at (0,-1)(w2){};
      \node[v] at (0,-2)(w3){};
      \node [v] at (1,.5)(v1){};
      \node[v] at (1,-.5)(v2){};
      \draw [dashed] (v1)--(w0)[bend right=15] to (w1)--(v2)--(w2)[bend right=15] to (w3)--(v1);
      \draw (v1)--(w1)[bend right=15] to (w0)--(v2)--(w3)[bend right=15] to (w2)--(v1);
    \end{tikzpicture}
\caption{A decomposition of a graph into two $6$-cycles in the proof of Lemma~\ref{2-cyclejoin}.}
\label{fig:2cyclejoin2}
\end{figure}
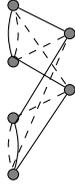

We finish by considering the case that $G_1$ has exactly two $2$-cycles $C_1$ and $C_2$.  By the previous case, we may assume that both $C_1$ and $C_2$ are odd.
We first handle the subcase that $G_1$ has no isolated vertices.  In this case, Figure~\ref{fig:2cyclejoin2} shows a decomposition of $E(G)$ into two $6$-cycles. 
Again, because of the odd $2$-cycle $C_1$, we can ensure that both $6$-cycles are even.  So, we may assume that $G_1$  has exactly two isolated vertices $x$ and $y$.  For $i \in \{1,2\}$, define $H_i$ to be $C_i$ together with all edges between $V(C_i)$ and $V(G_2)$. Define $H_3$ to be the $4$-cycle induced by $\{x,y\} \cup V(G_2)$. Note that $H_1$ and $H_2$ are strongly even-cycle decomposable by the previous case, and $H_3$ is obviously strongly even-cycle decomposable.  By the pigeonhole principle, there exist $i \neq j$ such that $\abs{E(H_i) \cap \Sigma} \equiv \abs{E(H_j) \cap \Sigma} \pmod{2}$.  But now we are done since $E(H_1) \cup E(H_2) \cup E(H_3)$ is a partition of $E(G)$.
\end{proof}

\begin{LEM}\label{lem:c42c}
    The join of a $4$-cycle and a $2$-cycle is strongly even-cycle decomposable.
\end{LEM}
\begin{proof}
    Let $G$ be the join of a $4$-cycle $C$ and a $2$-cycle $D$.
    Let $\Sigma$ be a signing of $E(G)$ with $\abs\Sigma$ even. If the $2$-cycle $D$ is even, then $(G,\Sigma)$ is even-cycle decomposable because $G\setminus E(D)$ is strongly even-cycle decomposable by Lemma~\ref{apex}. Thus, we may assume that $D$ is odd. Similarly, we may assume that $C$ is odd by Lemma~\ref{2-cyclejoin}. By Lemma~\ref{Km,nsign}, there is an even $4$-cycle $F$ in $G\setminus (E(C)\cup E(D))$. 
    Observe that $C\cup D\cup F$ can be decomposed into two even cycles and the remaining edges are even-cycle decomposable by Lemma~\ref{Km,nC4}.
\end{proof}
\begin{LEM}\label{co-claw}
 The join of two co-claws is strongly even-cycle decomposable.
\end{LEM}

\begin{proof}
Let $G$ be the join of two co-claws. Let $C_1$ and $C_2$ be the triangles of the two co-claws. Let $\Sigma$ be a signing of $E(G)$ with $\abs{\Sigma}$ even.
We may assume that $C_1$ and $C_2$ are both odd,  else we can apply Lemma~\ref{apex}. Let $x_1$ and $y_1$ be distinct vertices in $C_1$.
By an easy parity argument, there are vertices $x_2$ and $y_2$ in $C_2$ such that $C:=x_1x_2y_1y_2$ is an even $4$-cycle. Then $C\cup C_1\cup C_2$
can be decomposed into two even cycles. The remaining edges of $G$ can be decomposed into even cycles by Lemma~\ref{Km,nC4}.
\end{proof}

\begin{LEM}\label{lem:bowtie}
    The graph in Figure~\ref{fig:bowtie} is strongly even-cycle decomposable.
\end{LEM}

\begin{figure}
\centering
    \begin{tikzpicture}[scale=.75]
      \tikzstyle{v}=[circle,draw,fill=black!50,inner sep=0pt,minimum width=4pt]
      \node[v,label=left:$x$] at (0,0) (v) {};
      \node [v,label=left:$a$] at (0,1) (w1){};
      \node[v,label=left:$b$] at (0,-1)(w2){};
      \node [v,label=right:$y$] at (1,.5)(v1){};
      \node[v] at (1,-.5)(v2){};
      \foreach \x in {1,2} {
        \draw (v1)--(w\x)--(v2);
          \draw (v) [bend left=15] to (w\x) ;
          \draw (v) [bend right=15] to (w\x);
        }
    \end{tikzpicture}
\caption{$K_{2,3}$ with added two parallel edges incident with a degree-$2$ vertex.}
\label{fig:bowtie}
\end{figure}
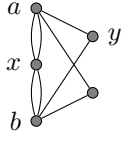
\begin{proof}
    Let $G$ be the graph in Figure~\ref{fig:bowtie}.
    Let $\Sigma$ be a signing of $E(G)$ with $\abs{\Sigma}$ even. 
    If both $2$-cycles in $G$ are even, then trivially we can decompose $E(G)$ into three even cycles.
    Thus by symmetry, we may assume that the cycle $ax$ is odd. Thus we may choose an even cycle $C=axby$ by selecting one of two edges joining $ax$ of the correct parity. Then both $C$ and $G \setminus E(C)$ are even cycles.
\end{proof}

\begin{LEM}\label{lem:k7c4c3}
    The graph $K_7-C_4-C_3$ in Figure~\ref{fig:k7c4c3} is strongly even-cycle decomposable.
\end{LEM}

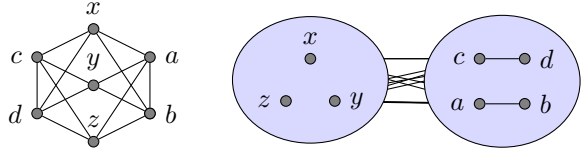
\begin{figure}
\centering
    \begin{tikzpicture}[scale=.75]
      \tikzstyle{v}=[circle,draw,fill=black!50,inner sep=0pt,minimum width=4pt]
      \node[v,label=$y$] at (0,0) (w0) {};
      \node [v,label=$x$] at (0,1) (w1){};
      \node[v,label=$z$] at (0,-1)(w2){};
      \node [v,label=right:$a$] at (1,.5)(v1){};
      \node [v,label=right:$b$] at (1,-.5)(v2){};
      \node [v,label=left:$c$] at (-1,.5)(v3){};
      \node [v,label=left:$d$] at (-1,-.5)(v4){};
      \node[v] at (1,-.5)(v2){};
      \foreach \x in {0,1,2} {
        \draw (v1)--(w\x)--(v2);
        \draw (v3)--(w\x)--(v4);
        }
        \draw(v1)--(v2); 
        \draw (v3)--(v4);
    \end{tikzpicture}
    $\quad$
 \begin{tikzpicture}[scale=.75,    every fit/.style={draw,fill=blue!15,ellipse,text width=35pt}]
    \tikzstyle{v}=[circle,draw,fill=black!50,inner sep=0pt,minimum width=4pt]
    \node[v,label=left:$z$] at (120+90:.5) (v1) {};
    \node[v,label=right:$y$] at (120*2+90:.5) (v2) {};
    \node[v,label=$x$] at (90:.5) (v3) {};
    \begin{scope}[xshift=3cm]
        \node [v,label=left:$a$] at (0,-.3) (w1) {};
        \node [v,label=right:$b$] at (.8,-.3) (w2) {};
        \node [v,label=left:$c$] at (0,.5) (w3) {};
        \node [v,label=right:$d$] at (.8,.5) (w4) {};
        \draw (w1)--(w2);
        \draw (w3)--(w4);
    \end{scope}
      \begin{scope}[on background layer]
      \foreach \x in {1,2,3} {
      \foreach \y in {1,2,3,4} 
        \draw (w\y)--(v\x);
        }
      \node[fit={(v1) (v2)(v3) }] (x) {};
      \node[fit={(w1) (w2)(w3)(w4) }] (y) {};
      \end{scope}
    \end{tikzpicture}
\caption{Two drawings of $K_7-C_4-C_3$.}
\label{fig:k7c4c3}
\end{figure}
\begin{proof}
    See Figure~\ref{fig:k7c4c3} for the labels of vertices of $G=K_7-C_4-C_3$.  Let $\Sigma$ be a signing of $E(G)$ with $\abs{\Sigma}$ even. %
    Since the symmetric difference of the three $4$-cycles $xcya$, $ycza$, and $xcza$ is empty, we may assume that $xcya$ is even. 
    
    If the $4$-cycle $xdyb$ is odd, then either $xdzb$ or $ydzb$ is even. By symmetry between $x$ and $y$, we may assume that the $4$-cycle $xdzb$ is even. Then $\{xcya, xdzb, ydczab\}$ is an even-cycle decomposition. Thus, we may assume that $xdyb$ is even.
    If the $3$-cycle $cdz$ is even, then the other $3$-cycle $abz$ is even and so $\{xcya,xdyb,cdz,abz\}$ is an even-cycle decomposition. Thus, we may assume that  $cdz$ is an odd $3$-cycle. Then either $C=dzby$ or $C'=dczby$ is an even cycle. If $C$ is even, then $\{xcya, dzby,dczabx\}$ is an even-cycle decomposition. If $C'$ is even, then $\{xcya, dczby, dzabx\}$ is an even-cycle decomposition. 
\end{proof}

\begin{LEM} \label{bowtiejoin}
Let $G_1$ be two $2$-cycles or two triangles meeting at a vertex $u$ and let $G_2=\overline{K_{2n}}$ for some $n \geq 1$.  Let $G$ be the graph obtained from the disjoint union of $G_1$ and $G_2$ by joining each vertex of $V(G_1) \setminus \{u\}$ to each vertex of $V(G_2)$.  Then $G$ is strongly even-cycle decomposable.  
\end{LEM}

\begin{proof}
Let  $\Sigma$ be a signing of $E(G)$ with $\abs{\Sigma}$ even.  By Lemma~\ref{lem:parityreduction}, we may assume $\abs{V(G_2)}=2$.  If $G_1$ is two $2$-cycles meeting at $u$, then $G$ is strongly even-cycle decomposable by Lemma~\ref{lem:bowtie}. So, we may assume that $G_1$ is two triangles meeting at $u$. Then $G$ is isomorphic to $K_7-C_4-C_3$ in Figure~\ref{fig:k7c4c3} and the conclusion follows from Lemma~\ref{lem:k7c4c3}.
\end{proof}

\begin{LEM}\label{lem:23cycle}
    Let $G$ be a graph with one $2$-cycle and one triangle meeting at a vertex $c$. Then the join of $G$ and $\overline{K_{2n}}$ is strongly even-cycle decomposable for all $n \geq 1$.
\end{LEM}
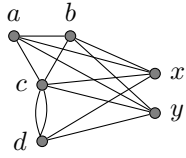
\begin{figure}
\centering
    
    \begin{tikzpicture}[scale=.75,    every fit/.style={draw,fill=blue!15,ellipse,text width=35pt}]
        \tikzstyle{v}=[circle,draw,fill=black!50,inner sep=0pt,minimum width=4pt]
        \node[v,label=left:$c$] at(0,0) (v0){};
        \node[v,label=$b$] at(60:1) (v1){};
        \node[v,label=$a$] at(120:1) (v2){};
        \node[v,label=left:$d$] at (0,-1) (v3){};
        \draw (v0)--(v1)--(v2)--(v0)[bend right=20] to (v3) [bend right=20] to (v0);
        \node[v,label=right:$y$]  at (2,-.5) (w1){};
        \node[v,label=right:$x$]  at (2,.2) (w2){};
        \foreach \x in {0,1,2,3}{
        \foreach \y in {1,2} {
        \draw (v\x)--(w\y);
        }}
    \end{tikzpicture}
\caption{The join of $\overline{K_{2}}$ with a graph having one $2$-cycle and one triangle sharing a vertex.}
\label{fig:23cycle}
\end{figure}
\begin{proof}
    Let $G'$ be the join of $G$ and $\overline{K_{2n}}$. Let $\Sigma$ be a signing of $E(G')$ with $\abs{\Sigma'}$ even. By Lemma~\ref{lem:parityreduction}, we may assume that $n=1$. We use the labels in Figure~\ref{fig:23cycle} for the vertices of $G$.

    Let $C_1$ be the $3$-cycle $abc$, $C_2$ be the $4$-cycle $bcdx$, and $C_3$ be the $5$-cycle $abcdx$. Since $G\setminus E(C_1)$ is strongly even-cycle decomposable by Lemma~\ref{2-cyclejoin}, we may assume that $C_1$ is odd. Since $G\setminus E(C_2)$ is a subdivision of $K_3+e_1+e_2+e_3$ where $E(K_3)=\{e_1,e_2,e_3\}$, which is strongly even-cycle decomposable, we may assume that $C_2$ is odd. 
    
    Since $E(C_3)=E(C_1)\Delta E(C_2)$, $C_3$ is even. Furthermore $G\setminus E(C_3)$ is a subdivision of $C_4+2e$ for an edge $e$ in $C_4$, which is strongly even-cycle decomposable by Lemma~\ref{parallel}.
\end{proof}

\section{Substituting a vertex} \label{sec:sub}

Let $G$ and $H$ be simple graphs and let $v \in V(G)$.  The \emph{substitution of $v$ by $H$ in $G$}
is the graph obtained
from the disjoint union of $G -v $ and $H$ by 
adding an edge $xy$ for all $x\in V(H)$ and all neighbors $y$ of $v$.

\begin{THM} \label{sub}
    Let $G$ be a simple strongly even-cycle decomposable graph, let $v$ be a non-isolated vertex of $G$, and let $H$ be a simple Eulerian graph with an odd number of vertices.  Then the substitution of $v$ by $H$ in $G$ is strongly even-cycle decomposable, provided that  $H$ is not $K_3$ or $\deg_G(v)\ge 4$.
\end{THM}

\begin{proof}
Let $G'$ be the substitution of $v$ by $H$ in $G$, and $\Sigma$ be a signing of $E(G')$ with $\abs{\Sigma}$ even.   Let $N$ be the set of neighbors of $v$.  We proceed by induction on $\abs{E(H)}+\abs{V(H)}$.

For us, a \emph{block} of $H$ is a maximal subgraph $B$ of $H$ such that $\abs{V(B)}\geq 2$ and $B -x$ is connected
for all $x \in V(B)$. If $(H, E(H) \cap \Sigma)$ contains an even cycle $C$, then we can apply the induction hypothesis to $H \setminus E(C)$.    Thus, we may assume each block of $(H, E(H) \cap \Sigma)$ is an odd cycle or a $K_2$. Since $H$ is Eulerian, no block can be a $K_2$. A \emph{leaf block} of $H$ is a block of $H$ that contains at most $1$ cut vertex of $H$.  

For each leaf block $B$ of $H$, we define $W_B$ to be the set of non-cut vertices of $H$ contained in $V(B)$.  
For distinct vertices $x, y \in W_B$, we define $G'(x,y)$ to be the union of $B$ together with all edges between $\{x,y\}$ and $N$.  We let $\overline{G'}(x,y)=G' \setminus E(G'(x,y))-x-y$ and $\overline{H}(x,y)=H\setminus E(B)-x-y$. Observe that $G'(x,y)$ is a subdivision of the join of a $2$-cycle and $\overline{K_{\abs{N}}}$.   By Lemma~\ref{2-cyclejoin}, $G'(x,y)$ is strongly even-cycle decomposable.
Furthermore $\overline{G'}(x,y)$ is obtained from $G$ by substituting $v$ with $\overline{H}(x,y)$. 

For each $t \in V(H)$, define $G_t:=G' - V(H - t)$ and observe that $G_t$ is isomorphic to $G'$. We also define $p(t)$ to be the
number of edges of $\Sigma$ between $t$ and $N$.  

Suppose $\abs{E(G'(x,y)) \cap \Sigma}$ is even for some leaf block $B$ of $H$ and $x,y\in W_B$. Then we are done by induction, unless $\abs{N}=2$ and $\overline{H}(x,y)=K_3$. That is, $H$ is the union of two triangles meeting at a vertex $c$.  If $\abs{E(G_t) \cap \Sigma}$ is even for some $t \in V(H) \setminus\{c\}$, then we are done since $G_t$ is 
strongly-even cycle decomposable, and the remaining edges are even-cycle decomposable by Lemma~\ref{lem:23cycle}.
Thus, we may assume $p(u)\equiv p(w)\pmod 2$ for all $u,w \in V(H) \setminus \{c\}$.
Define $H^+$ to be $H$ together with all edges between $V(H) \setminus \{c\}$ and $N$.  Observe that $\abs{E(H^+) \cap \Sigma}$ is even, and that
$H^+$ is strongly even-cycle decomposable by Lemma~\ref{bowtiejoin}.  Since $G' \setminus E(H^+)$ is isomorphic to $G'$ together with 4 isolated vertices, we are done.

We may hence assume $\abs{E(G'(x,y)) \cap \Sigma}$ is odd for all leaf blocks $B$ of $H$ and $x,y\in W_B$.  Suppose that $H$ contains distinct leaf blocks $B_1$ and $B_2$.  Let $a,b,c,d$ be distinct vertices with $a,b \in W_{B_{1}}$ and $c,d \in W_{B_2}$.  Let $G'(a,b,c,d):=G'(a,b) \cup G'(c,d)$ and let $\overline{G'}(a,b,c,d)=G' \setminus E(G'(a,b,c,d))-a-b-c-d$. Since $\abs{E(G'(a,b)) \cap \Sigma}$ and $\abs{E(G'(c,d)) \cap \Sigma}$ are both odd, $\abs{E(G'(a,b,c,d)) \cap \Sigma}$ is even.  Note that $G'(a,b, c,d)$ is a subdivision of a graph that is strongly even-cycle decomposable by either Lemma~\ref{2-cyclejoin} or Lemma~\ref{bowtiejoin}, and so we are done by the induction hypothesis unless $\abs N=2$ and $H$ is the windmill graph (see Figure~\ref{fig:windmill}) or $\abs N=2$ and $H$ is a path of three triangles (see Figure~\ref{fig:3triangle}).

\begin{figure}
\centering
    \begin{tikzpicture}[scale=.75]
    \tikzstyle{v}=[circle,draw,fill=black!50,inner sep=0pt,minimum width=4pt]
    \node [v] at (0,0) (v0) {};
    \foreach \x in {1,2,3,4,5,6}  {
        \node[v] at (60*\x:1) (v\x) {};
        \draw (v0)--(v\x);
    }
    \draw (v1)--(v2);
    \draw (v3)--(v4);
    \draw (v5)--(v6);
    \end{tikzpicture}
\caption{The windmill graph.}
\label{fig:windmill}
\end{figure}
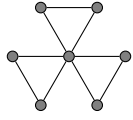

\medskip
Suppose $H$ is the windmill graph.  By the pigeonhole principle, there are  distinct blocks  $B_1$ and $B_2$ of $H$ and $x_1 \in W_{B_{1}}$ and $x_2 \in W_{B_{2}}$  such that $p(x_1)\equiv p(x_2)\pmod 2$. Let $G''$ be the subgraph of $G'$ consisting of $B_1 \cup B_2$ together with all edges between $\{x_1, x_2\}$ and $N$.  Note that $G''$ contains and even number of odd edges . We are done since $G''$ is strongly even-cycle decomposable by Lemma~\ref{bowtiejoin}, while the remaining edges are  even-cycle decomposable by the induction hypothesis.

\begin{figure}
\centering
    \begin{tikzpicture}[scale=.75]
    \tikzstyle{v}=[circle,draw,fill=black!50,inner sep=0pt,minimum width=4pt]
    \node [v,label=below:$x_1$] at (0,0) (v0) {};
        \node[v,label=below:$u$] at (1,0) (v1) {};
        \node [v,label=$y_1$] at (1-.5,.5) (w1){};
        \node[v,label=below:$x_2$] at (2,0) (v2) {};
        \node [v,label=$y_2$] at (2-.5,.5) (w2){};
        \node[v] at (3,0) (v3) {};
        \node [v] at (3-.5,.5) (w3){};
    \draw (v0) --(v1)--(v2)--(v3);
    \draw (v0)--(w1)--(v1)--(w2)--(v2)--(w3)--(v3);
    
    \end{tikzpicture}
\caption{A path of three triangles.}
\label{fig:3triangle}
\end{figure}
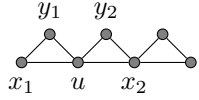

\medskip
Suppose $H$ is a path of three triangles. Let $B_1:=x_1y_1u$ be a leaf block and $B_0:=ux_2y_2$ be a triangle of $H$ sharing a vertex $u$ with $B_1$ (See Figure~\ref{fig:3triangle}).  Let $H'$ be the subgraph of $H$ consisting of $E(B_1) \cup E(B_0)$ together with all edges between $\{x_1,y_1,x_2,y_2\}$ and $N$.  If $p(x_2)\equiv p(y_2)\pmod 2$, then $H'$ contains an even number of odd edges.  Moreover, $H'$ is strongly even-cycle decomposable by Lemma~\ref{bowtiejoin}, and the remaining edges are even-cycle decomposable by the induction hypothesis.  Thus, we may assume $p(x_2) \not\equiv p(y_2)\pmod2$.  It follows that we may pick $z \in \{x_2, y_2\}$ such that $p(x_1)\equiv p(z)\pmod 2$.  Let $G''$ be the subgraph of $G'$ consisting of $E(B_1) \cup E(B_0)$ 
together with all edges between $\{x_1, z\}$ and $N$.  Note that $G''$ contains an even number of odd edges.  We are thus finished since $G''$ is strongly even-cycle decomposable by Lemma~\ref{bowtiejoin}, and the remaining edges are even-cycle decomposable by the induction hypothesis.

We have therefore reduced to the case that $H$ has at most one leaf block. By Lemma~\ref{lem:parityreduction}, we may assume that $H$ has at most $2$ isolated vertices. So, we may assume that $H$ is a cycle $C$ together with a set $S$ of isolated vertices with $\abs{S}\le 2$.

If $\abs{V(H)} \geq 7$, then then there exist two vertices
$x,y \in V(C)$ such that $p(x)\equiv p(y)\pmod 2$. 
Since $K_{2,\abs{N}}$ is strongly even-cycle decomposable by Lemma~\ref{Km,n}, we may shrink $H$ by suppressing $x$ and $y$ and apply the induction hypothesis
(as $\abs{V(H)} \geq 7$).

We may hence assume that $\abs{V(H)} \leq 5$, and so $H$ is a $5$-cycle, a $4$-cycle with one isolated vertex, a triangle with two isolated vertices,  or a triangle.  
We finish by handling each of these cases separately. 

\medskip
Suppose $H$ is a $5$-cycle. Let $x$ and $y$ be distinct vertices of $H$ such that $p(x)\equiv p(y)\pmod 2$.  Then we may shrink $H$ by suppressing $x$ and $y$
and apply the induction hypothesis, unless $\abs N=2$.  So, we may assume $\abs N=2$.  Let $u$ be an arbitrary vertex of $H$. By Lemma~\ref{almost}, there is an almost even-cycle decomposition $\{C'\} \cup \mathcal{C}$ of $(G_u,\Sigma\cap E(G_u))$, where all cycles in $\mathcal{C}$ are even and $u \in V(C')$.  
Let $G''$ be the graph induced by the edges in $E(G') \setminus (E(G_u)\setminus E(C'))$. Observe that
$G''$ is isomorphic to a subdivision of the join of a $5$-cycle and $K_2$.  By Lemma \ref{cliquejoin}, $G''$ is strongly even-cycle decomposable.

\medskip 
Suppose $H$ is a $4$-cycle with an isolated vertex $s$.   If $\abs{E(G_s)\cap \Sigma}$ is even, then we are done since the graph induced by the remaining edges is the join of $C_4$ and $\overline{K_{\abs{N}}}$, which is strongly even-cycle decomposable by Lemma~\ref{apex}.  Therefore we may assume that $\abs{E(G_s)\cap \Sigma}$ is odd. 

Suppose $\abs{N} \geq 4$.  If $p(s)\equiv p(x)\pmod{2}$ for some $x \in V(C)$, then the set $F$ of edges from $\{s,x\}$ to $N$ is even-cycle decomposable. Since $\abs{N}\ge 4$, the remaining edges are even-cycle decomposable by the induction hypothesis.
Therefore we may assume $\abs{E(G_x)\cap \Sigma}$ is even for all $x \in V(C)$.  Then for every vertex $x$ of $C$, the graph obtained from $G'\setminus E(G_x)$ by suppressing $x$ is strongly even-cycle decomposable by Lemma~\ref{apex}.

It remains to consider the case when $\abs{N}=2$.  Let $\{C'\} \cup \mathcal{C}$ be an almost even-cycle decomposition of $(G_s,\Sigma\cap E(G_s))$, where all cycles in $\mathcal{C}$ are even and $s \in V(C')$.  Let $G''$ be the graph induced by $E(G') \setminus (E(G_s)\setminus E(C'))$. Note that $G''$ is a subdivision of the join of $C_4$ and a $2$-cycle, which is strongly even-cycle decomposable by Lemma~\ref{lem:c42c}. 

\medskip 
Suppose $H$ is a triangle with two isolated vertices.  Let $S=\{x,y\}$. Suppose $\abs N \geq 4$.  
We first handle the subcase that $p(x)\equiv p(y)\pmod2$.  Here, we shrink $H$ by suppressing $x$ and $y$ and apply the induction hypothesis to deduce that $(G',\Sigma)$ is even-cycle decomposable since 
$\abs {N} \geq 4$.  So, we may assume that $p(x) \not\equiv p(y)\pmod2$.  By symmetry, we may assume that $G_x$ contains an even number of odd edges.   
Let $G''$ be the graph induced by  $E(G') \setminus E(G_x)$.
Observe that $G''$ is the join of a co-claw and $\overline{K_{\abs N}}$, and we are done by Lemma~\ref{apex}.

We may hence assume that $\abs N =2$.  Let $\{C'\} \cup \mathcal{C}$ be an almost even-cycle decomposition of $(G_x,\Sigma\cap E(G_x))$, where all cycles in $\mathcal{C}$ are even and $x \in V(C')$. Let $G''$ be the graph induced by $E(G') \setminus (E(G_x)\setminus E(C'))$.  Observe that $G''$ is a subdivision of $K_5+2e$, which is strongly even-cycle decomposable by Lemma~\ref{K5+}.

\medskip
Suppose $H$ is a triangle $abc$.  By hypthosis, $\abs{N} \geq 4$.  Let $\{C'\} \cup \mathcal{C}$ be
an almost even-cycle decomposition of $(G_a,\Sigma\cap E(G_a))$, where all cycles in $\mathcal{C}$ are even and $a \in V(C')$. Let $G''$ be the graph obtained from $G'\setminus (E(G_a)\setminus E(C'))$ by deleting isolated vertices and suppressing degree-$2$ vertices not in $N$.
It is enough to prove that $G''$ is strongly even-cycle decomposable. Notice that $G''-b-c$ is a simple graph with exactly one cycle and $\abs{V(G''-b-c)} \geq 5$. Since $G''$ is the join of $K_2$ and $G''-b-c$, $G''$ is strongly even-cycle decomposable by Lemma~\ref{cliquejoin}.
\end{proof}

\section{Substituting twin vertices} \label{sec:twinsub}

If $u$ and $v$ are twins of $G$, and $H$ is a simple graph, then
the \emph{twin substitution of $\{u,v\}$ by $H$ in $G$} is the graph obtained
from the disjoint union of $G - u-v$ and $H$ by 
adding an edge $xy$ for all $x\in V(H)$ and all neighbors $y$ of $v$.

\begin{THM} \label{twin}
    Let $G$ be a simple strongly even-cycle decomposable graph and let $u$ and $v$ be non-isolated, non-adjacent twin vertices of $G$. If $H$ is a simple Eulerian graph with an even number of vertices, then the twin substitution of $\{u,v\}$ by $H$ in $G$ is also strongly even-cycle decomposable, provided that $H$ is not a co-claw or $\deg_G(v)\ge 4$.
\end{THM}
\begin{proof}
    We proceed by induction on $\abs{E(H)}$. If $H$ has an isolated vertex, then it follows easily from Theorem~\ref{sub}. Thus, we may assume that $H$ has no isolated vertices.

    Let $G'$ be the twin substitution of $\{u,v\}$ by $H$ in $G$. We may assume that $G$ has no isolated vertices and $\abs{V(H)}\ge 4$.  Let $N$ be the set of neighbors of $v$ in $G$.  Let $\Sigma$ be a signing of $E(G')$ with $\abs{\Sigma}$ even. If $(H,E(H)\cap \Sigma)$ has an even cycle $C$, then we apply the induction hypothesis with $H'=H\setminus E(C)$.  Thus we may assume that $(H,E(H)\cap \Sigma)$ has no even cycles. Since $H$ is Eulerian, every block of $H$ is an odd cycle.
    
    For each leaf block $B$ of $H$, we define $W_B$ to be the set of non-cut vertices of $H$ contained in $V(B)$. For distinct vertices $x,y\in W_B$, we define $G'(x,y)$ to be a subgraph of $G'$ that is the union of $B$ together with all edges between $\{x,y\}$ and $N$. We let $\overline {G'} (x,y)=G'\setminus E(G'(x,y))-x-y$. Observe  $G'(x,y)$ is a subdivision of the join of a $2$-cycle and $\overline{K_{\abs N}}$. By Lemma~\ref{2-cyclejoin}, $G'(x,y)$ is strongly even-cycle decomposable. Thus we may assume for all such choices $x$, $y$ of $B$, $\abs{E(G'(x,y))\cap \Sigma}$ is odd. 
    
    If $H$ has at least two odd cycles, then let $B_1$, $B_2$ be distinct leaf blocks of $H$ and let $a$, $b$, $c$, $d$ be distinct vertices with $a,b\in W_{B_1}$ and $c,d\in W_{B_2}$. Let $G'(a,b,c,d)=G'(a,b)\cup G'(c,d)$ and let $\overline{G'}(a,b,c,d)=G\setminus E(G'(a,b,c,d))-a-b-c-d$. Since $\abs{E(G'(a,b))\cap \Sigma}$ and $\abs{E(G'(c,d))\cap\Sigma}$ are both odd, $\abs{E(G'(a,b,c,d))\cap \Sigma}$ is even. Note that $G'(a,b,c,d)$ is a subdivision of a graph that is strongly even-cycle decomposable by Lemma~\ref{2-cyclejoin} or Lemma~\ref{bowtiejoin}. Furthermore, $\overline{G'}(a,b,c,d)$ is strongly even-cycle decomposable by the induction hypothesis. Thus, we may assume  $(H,\Sigma\cap E(H))$ is an odd cycle of even length.

Let $x$ and $y$ be two non-adjacent vertices in $H$.
Let $G_1=G'-(V(H)\setminus\{x,y\})$ and $G_2=G'\setminus E(G_1)\setminus E(H)$. Note that $G_1$ is isomorphic to $G$, and $E(G')=E(G_1)\cup E(G_2)\cup E(H)$.

Suppose  $\abs{E(G_1)\cap \Sigma}$ is even and hence  $\abs{E(G_2)\cap \Sigma}$ is
odd. By Lemma~\ref{Km,nsign}, there exists an odd $4$-cycle $C$ in
$(G_2,E(G_2)\cap \Sigma)$. Then $C\cup H$ can be decomposed into two even cycles, and by Lemma~\ref{Km,nC4},
$G_2\setminus E(C)$ is strongly even-cycle decomposable.

Hence $\abs{E(G_1)\cap \Sigma}$ is odd.
Thus there is an almost
even-cycle decomposition $\mathcal{C}$ of $G_1$ such that $x$ belongs
to an odd cycle $C_1\in \mathcal{C}$.

If  $\abs{V(H)}\ge 6$, then let $F$ be the set of edges from $\{x,y\}$ to $N$. 
As  $\abs{G'(x,y)\cap \Sigma}$ is odd, $\abs{F\cap \Sigma}$ is even and so $F$ is even-cycle decomposable by Lemma~\ref{basebasecase}. The remaining edges are even-cycle decomposable by the induction hypothesis applied to the graph obtained from $H$ by suppressing $x$ and $y$. Therefore we may assume  $\abs{V(H)}=4$ and $H$ is a $4$-cycle.

Suppose $\abs N\geq 4$. 
Let $G_1'$ be the spanning subgraph of $G_1$ on $V(G_1)$ whose edge set is precisely $E(C_1)$. Let $G_1''$ be the graph obtained from $G_1'$ by deleting all isolated vertices and suppressing all degree-$2$ vertices not in $N\cup\{x,y\}$ so that $V(G_1'')=N\cup\{x,y\}$.
Let $G''$ be the graph induced by $E(G')\setminus (E(G_1)\setminus E(C_1))$. 
Since $G_1''$ is a simple graph with exactly one cycle and $\abs{V(G_1'')}\geq 6$, $G''$ is strongly even-cycle decomposable by Lemma~\ref{apex} because $G''$ is a subdivision of the join of $\overline{K_2}$ and $G_1''$.

Thus we may assume that $\abs N=2$.  Let $C_2$ be the cycle in $\mathcal{C}$ containing $y$
(possibly $C_1=C_2$). Then $C_1\cup C_2$ can be decomposed into
two $xy$-paths, and thus $C_1\cup C_2\cup H$ can be decomposed 
into two even cycles. 
\end{proof}

\begin{THM}\label{thm:twin2}
    Let $G$ be a simple strongly even-cycle decomposable graph and let $u$ and $v$ be adjacent twin vertices of $G$ having degree at least $4$. If $H$ is a simple anti-Eulerian graph, then the twin substitution of $\{u,v\}$ by $H$ in $G$ is also strongly even-cycle decomposable.
\end{THM}
\begin{proof}
    Let $G'$ be the twin substitution of $\{u,v\}$ by $H$ in $G$.
    We proceed by induction on $\abs{E(H)}$. Let $\Sigma$ be a signing of $E(G')$ with $\abs\Sigma$ even. Let $N=N_G(v) \setminus \{u\}$. By assumption, $\abs{N}\ge 3$ and $\abs{N}$ is odd.

    If $(H,E(H)\cap \Sigma)$ has an even cycle $C$, then we apply the induction hypothesis to $G$ and $H\setminus E(C)$. Thus, we may assume that $(H,E(H)\cap \Sigma)$ has no even cycle. Therefore every block of $H$ is either an odd cycle or a $K_2$. 
    
    If $B$ is a leaf block of $H$, then $B=K_2$ because $H$ is anti-Eulerian. Since $B$ has at most one cut vertex, each leaf block has at least one vertex of degree $1$. Conversely, every vertex of degree $1$ in $H$ is in a leaf block.
    
    If $H$ has an odd cycle $C$, then there exist vertices $x$, $y$ of degree $1$ in $H$ such that there are two paths $P_1$ and $P_2$ from $x$ to $y$ having different parity. For $i=1,2$, let $G_i$ be the subgraph of $G$ induced by all edges between $\{x,y\}$ and $N$ and all edges in $P_i$. Then there exists $i\in \{1,2\}$ such that $\abs{E(G_i)\cap \Sigma}$ is even. Since $G_i$ is a subdivision of the join of $\overline{K_{\abs N}}$ with $K_2$, $G_i$ is strongly even-cycle decomposable by Lemma~\ref{cliquejoin}. It remains to show that $G'\setminus E(G_i)$ is strongly even-cycle decomposable, which is implied by the induction hypothesis with $G$ and $H\setminus E(P_i)-x-y$. 
    
    Therefore, we may assume that $H$ has no cycles. Suppose that $H$ has a component with three leaves $v_1$, $v_2$, $v_3$. For each $1\le i<j\le 3$, let $P_{ij}$ be the unique path from $v_i$ to $v_j$ in $H$. Let $G_{ij}$ be the subgraph of $G$ induced by all edges between $\{v_i,v_j\}$ and $N$ and all edges in $P_{ij}$. Since $E(G_{12})\Delta E(G_{23})\Delta E(G_{31})=\emptyset$, there exist $1\le i<j\le 3$ such that $\abs{E(G_{ij})\cap \Sigma}$ is even.
    As before, $G_{ij}$ is strongly even-cycle decomposable by Lemma~\ref{cliquejoin} and $G'\setminus E(G_{ij})$ is strongly even-cycle decomposable by the induction hypothesis. 
    
    Thus, we may assume that each component of $H$ is isomorphic to $K_2$. Let $H_1,\ldots, H_m$ be the components of $H$. If $m=1$, then $G'$ is isomorphic to $G$. So, we may assume $m \geq 2$.
    For each $1\le i\le m$, let $G_i'=G'- (V(H)\setminus V(H_i))$ and let $G_i''$ be the subgraph of $G'$ induced by $H_i$ and all edges incident with a vertex in $H_i$. Observe that $G_i''$ is isomorphic to the join of $K_{2}$ and $\overline{K_{\abs{N}}}$, which
    is strongly even-cycle decomposable by Lemma~\ref{cliquejoin}. If $\abs{E(G_i'')\cap \Sigma}$ is even, then we are done since $G'\setminus E(G_i'')$ is strongly even-cycle decomposable by the induction hypothesis.
    
    Thus, we may assume that $\abs{E(G_i'')\cap \Sigma}$ is odd for all $1\le i\le m$. Then $\abs{E(G_1''\cup G_2'')\cap \Sigma}$ is even. Since $K_7-C_4-C_3$ is strongly even-cycle decomposable by Lemma~\ref{lem:k7c4c3} and $K_7-C_4-C_3$ is the join of $\overline{K_{3}}$ with $2K_2$, $G_1''\cup G_2''$ is strongly even-cycle decomposable by Lemma~\ref{lem:parityreduction}. If $m>2$, then $G'\setminus E(G_1''\cup G_2'')$ is strongly even-cycle decomposable by the induction hypothesis and therefore $(G',\Sigma)$ is even-cycle decomposable. 
    
    Hence, we may assume that $m=2$.   Then $E(G_1')\cup E(G_2'')=E(G')$ and $E(G_1')\cap E(G_2'')=\emptyset$ and so $\abs{E(G_1')\cap \Sigma}$ is odd. 
    
    \begin{CLAIM} \label{fix}
    If $(G_1',  E(G_1') \cap \Sigma)$ has an almost even-cycle decomposition $\mathcal{C}$ such that the unique odd cycle $C \in \mathcal{C}$ uses at 
    least two vertices of $N$, then $(G',\Sigma)$ is even-cycle decomposable. 
    \end{CLAIM}
    
    \begin{subproof}[Subproof]
    Let $\mathcal{C}$ be an almost even-cycle decomposition of $(G_1',  E(G_1') \cap \Sigma)$  such that the unique odd cycle $C \in \mathcal{C}$ uses at 
    least two vertices of $N$. Let $G_1^*=(V(G_1'),E(C))$ be the spanning subgraph of $G_1'$ whose edge set is $E(C)$. It is enough to show that $G_1^*\cup G_2''$ is strongly even-cycle decomposable. 
    
    Let $G_1^{**}$ be the graph obtained from $G_1^*$ by removing isolated vertices not in $N$ and suppressing all vertices of degree $2$ not in $N$. Observe $G_1^{**}$ has exactly one cycle, $V(G_1^{**})=N$, $V(G_2'')=N \cup V(H_2)$, and  $G_1^*\cup G_2''$ is strongly even-cycle decomposable if $G_1^{**}\cup G_2''$ is strongly even-cycle decomposable. Also note $G_1^{**}\cup G_2''$ is the join of $G_1^{**}$ and $K_2$.  
    
    If $G_1^{**}$ does not have a $2$-cycle, then $G_1^{**}\cup G_2''$ is strongly-even cycle decomposable by Lemma~\ref{cliquejoin}. Thus, we may assume that $G_1^{**}$ has a $2$-cycle.  By subdividing the edge in $H_2$, we conclude that $G_1^{**}\cup G_2''$ is strongly even-cycle decomposable by Lemma~\ref{2-cyclejoin}.
    \end{subproof}
    
    Let $e=ab$ be the unique edge in $H_1$ and let $\mathcal{C}$ be an almost even-cycle decomposition of $(G_1',  E(G_1') \cap \Sigma)$ such that $e$ belongs to the unique odd cycle $C_1 \in \mathcal{C}$.  If $C_1$ uses at least two vertices of $N$, then we are done by Claim \ref{fix}.  So, we may assume that $C_1=abc$, with $c \in N$.  Let $C_2 \neq C_1$ be a cycle in $\mathcal{C}$ such that $a \in V(C_2)$. Note that $|E(C_1 \cup C_2) \cap \Sigma|$ is odd, since $C_2$ is even.  
    
    If $c \in V(C_2)$, then $C_1 \cup C_2$ can be decomposed into two cycles $C_1'$ and $C_2'$ such that $C_1'$ and $C_2'$ both use at least two vertices of $N$.  One of $C_1'$ and $C_2'$ is odd and the other is even, so we are done by Claim \ref{fix}. 
    
    Suppose $c \notin V(C_2)$ and $b \in V(C_2)$.  If $C_2$ is not a $4$-cycle, then $C_1 \cup C_2$ can be decomposed into two cycles $C_1'$ and $C_2'$ such that $C_1'$ and $C_2'$ both use at least two vertices of $N$, so we are again done by Claim \ref{fix}.  We may thus assume that $C_2$ is a $4$-cycle. Let
    $f=a'b'$ be the unique edge in $H_2$.  Consider $G_2^*=G_2'' \cup E(C_1) \cup E(C_2)$.  It suffices to show that $G_2^*$ is strongly even-cycle decomposable.  Let $G_2^{**}$ be the graph obtained from $G_2^*$ by removing isolated vertices and suppressing degree-$2$ vertices.  Note that there are $\abs{N}-2$ parallel edges between $a'$ and $b'$ in $G_2^{**}$.  Let $G_2^{***}$ be the graph obtained from $G_2^{**}$ by deleting $\abs{N}-3$ edges between $a'$ and $b'$.  By Lemma \ref{parallel} it suffices to show that $G_2^{***}$ is strongly even-cycle decomposable.  But $G_2^{***}=K_7-C_4-C_3$, so we are done by Lemma \ref{lem:k7c4c3}.

    The remaining case is  $b \notin V(C_2)$ and $c \notin V(C_2)$.  Again let $G_2^*=G_2'' \cup E(C_1) \cup E(C_2)$ and $f=a'b'$ be the unique edge in $H_2$. 
     Let $G_2^{**}$ be the graph obtained from $G_2^*$ by removing isolated vertices and suppressing degree-2 vertices.  It suffices to show that $G_2^{**}$ is strongly even-cycle decomposable.  Let $N^{**}$ be the vertices of $N$ that have not been suppressed. Let $\Sigma^{**}$ be a signing of $E(G_2^{**})$ with $|\Sigma^{**}|$ even.  By removing even $2$-cycles between $a'$ and $b'$, we may assume that there are $k$ edges between $a'$ and $b'$, where $k \in \{1,2\}$.  By removing even $4$-cycles between $\{a',b'\}$ and $N^{**} \setminus \{c\}$,
    and then suppressing degree-2 vertices, we may assume that $|N^{**}| \in \{2,3\}$.  Since $G_2^{**}$ is Eulerian, the only possibilities are $k=1$ and $|N^{**}|=3$ or $k=2$, and  $|N^{**}| =2$, see Figure~\ref{fig:last2}.

    \begin{figure}
\centering
      \begin{tikzpicture}[scale=.75,    every fit/.style={draw,fill=blue!15,ellipse,text width=35pt}]
        \tikzstyle{v}=[circle,draw,fill=black!50,inner sep=0pt,minimum width=4pt]
        \node[v,label=left:$a$] at(0,0) (v0){};
        \node[v] at(-60:1) (v1){};
        \node[v] at(-120:1) (v2){};
        \node[v,label=left:$c$] at (0,1) (v3){};
        \draw (v0)--(v1)--(v2)--(v0)[bend right=20] to (v3) [bend right=20] to (v0);
        \node[v,label=right:$b'$]  at (2,-.5) (w1){};
        \node[v,label=right:$a'$]  at (2,.2) (w2){};
        \draw (w1) -- (w2);
        \foreach \x in {1,2,3}{
        \foreach \y in {1,2} {
        \draw (v\x)--(w\y);
        }}
    \end{tikzpicture}$\quad$  
    \begin{tikzpicture}[scale=.75,    every fit/.style={draw,fill=blue!15,ellipse,text width=35pt}]
        \tikzstyle{v}=[circle,draw,fill=black!50,inner sep=0pt,minimum width=4pt]
        \node[v,label=left:$a$] at(0,0) (v0){};
        \node[v,label=$c$] at(0,1) (v1){};
        \node[v] at (0,-1) (v3){};
        \draw (v0)[bend right=20] to (v1)[bend right=20]to 
            (v0)[bend right=20] to (v3) [bend right=20] to (v0);
        \node[v,label=right:$b'$]  at (2,-.5) (w1){};
        \node[v,label=right:$a'$]  at (2,.2) (w2){};
        \draw (w1)[bend right=20] to (w2)[bend right=20]to 
            (w1);
        \foreach \x in {1,3}{
        \foreach \y in {1,2} {
        \draw (v\x)--(w\y);
        }}
    \end{tikzpicture}
\caption{Two remaining possibilities in the proof of Theorem~\ref{thm:twin2}.}\label{fig:last2}
\end{figure}
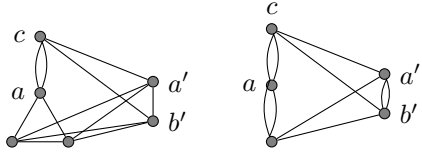

    Suppose $k=2$, and  $|N^{**}| =2$.  If the $2$-cycle $a'b'a'$ is even, then we can remove it and apply Lemma~\ref{bowtiejoin}.  So, we may assume that the
    $2$-cycle $a'b'a'$ is odd.  Observe that $E(G^{**})$ decomposes into two $5$-cycles.  We are thus done, since we can ensure that both these $5$-cycles are even because of the odd $2$-cycle $a'b'a'$.  
    
    Suppose $k=1$ and $|N^{**}|=3$.  If the $2$-cycle $aca$ is even, then by removing it we get a graph which is a subdivision of the join of two $2$-cycles.  Clearly, the join of two $2$-cycles is strongly even-cycle decomposable, so we may assume that $aca$ is odd.   Observe that $E(G^{**})$ decomposes into two $6$-cycles.  We are thus done, since we can ensure that both these $6$-cycles are even because of the odd $2$-cycle $aca$.  
\end{proof}

\section{Joins} \label{sec:join}
The main result of this section is that if $G$ is a simple Eulerian graph that is a join of two graphs $G_1$ and $G_2$ with $\abs{V(G_1)}, \abs{V(G_2)} \geq 2$, then $G$ is strongly even-cycle decomposable if and only if $G$ is neither $K_5$ nor $K_5$ with an edge subdivided.  

\begin{LEM} \label{lem:join-even}
Let $G_1$ and $G_2$ be simple Eulerian graphs with an even number of vertices.  Then the join of $G_1$ and $G_2$ is strongly even-cycle decomposable if and only if it is not $K_5$ with an edge subdivided. 
\end{LEM}
\begin{proof}
    The forward direction is obvious.  For the converse, let $G$ be the join of $G_1$ and $G_2$ and assume $G$ is not $K_5$ with an edge subdivided.
    If both $G_1$ and $G_2$ are either $\overline{K_2}$ or a co-claw, then $G$ is $K_{2,2}$ or the join of two co-claws, both of which are strongly even-cycle decomposable (the latter by Lemma~\ref{co-claw}). Thus we may assume that $G_1$ is neither $\overline{K_2}$ nor a co-claw. In particular, $\abs{V(G_1)}\ge 4$.  
Let $X=\{a,b\}$ and $Y=\{c,d\}$ be a bipartition of $K_{2,2}$.  Let $H$ be the twin substitution of
 $\{a,b\}$ by $G_1$ in $K_{2,2}$. Since $G_1$ is not a co-claw, $H$ is strongly even-cycle decomposable by Theorem~\ref{twin}.  Then $G$ is isomorphic to the twin substitution of
$\{c,d\}$ by $G_2$ in $H$, which  is strongly even-cycle decomposable by  Theorem~\ref{twin} because $\deg_H(c)\ge 4$.
\end{proof}

We also prove a variant of our previous lemma for anti-Eulerian graphs. 

\begin{LEM}\label{lem:join-odd}
Let $G_1$ be a simple Eulerian graph with $\abs{V(G_1)} \geq 3$ and $\abs{V(G_1)}$ odd.  Let $G_2$ be a simple anti-Eulerian graph with $\abs{V(G_2)} \geq 2$.  Then the join of $G_1$ and $G_2$ is strongly even-cycle decomposable if and only if $G_1 \neq K_3$ or $G_2 \neq K_2$.  
\end{LEM}

\begin{proof}
    The forward direction is obvious since the join of $K_3$ and $K_2$ is $K_5$.  For the converse, 
    let $G$ be the join of $G_1$ and $G_2$, where $G_1 \neq K_3$ or $G_2 \neq K_2$.  First suppose that $G_1 \neq K_3$. Let $G_3$ be the join of $G_1$ and $K_2$. Let $V(K_3)=\{a,b,c\}$. Since $G_3$ is isomorphic to the substitution of $a$ by $G_1$ in $K_3$, $G_3$ is strongly even-cycle decomposable by Theorem~\ref{sub}. Since $G$ is isomorphic to the twin substitution of $\{b,c\}$ by $G_2$ in $G_3$, $G$ is strongly even-cycle decomposable by Theorem~\ref{thm:twin2}.

    Hence we may assume that $G_1=K_3$ and $G_2\neq K_2$. In particular, $\abs{V(G_2)}\ge 4$. 
    Let $v\in V(G_1)$, $G_1'=G-(V(K_3)\setminus\{v\})$, and $G_2'=G_1-v$. Then $G$ is isomorphic to the join of $G_1'$ and $G_2'$,
    where $G_1' \neq K_3$ since $\abs{V(G_1')}\ge 5$. By the previous argument with $(G_1, G_2)$ replaced by $(G_1', G_2')$, we conclude that $G$ is strongly even-cycle decomposable. 
\end{proof}

Combining the previous two lemmas, we obtain the main result of this section. 

\begin{THM}\label{thm:join}
    Let $G$ be a simple Eulerian graph that is a join of two graphs each having at least two vertices. Then $G$ is strongly even-cycle decomposable if and only if $G$ is neither $K_5$ nor $K_5$ with an edge subdivided.
\end{THM}
\begin{proof}
    The forward direction is trivial. For the converse, suppose that $G$ is the join of two graphs $G_1$ and $G_2$ with $\abs{V(G_1)}, \abs{V(G_2)}\ge 2$. If both $G_1$ and $G_2$ have an even number of vertices, then both are Eulerian and we apply Lemma~\ref{lem:join-even}. 
    If $\abs{V(G_1)}$ is even and $\abs{V(G_2)}$ is odd, then $G_1$ is anti-Eulerian and $G_2$ is Eulerian and so  we apply Lemma~\ref{lem:join-odd}.
    If both $G_1$ and $G_2$ have an odd number of vertices, then $G_1$ is anti-Eulerian, contradicting the fact that every anti-Eulerian graph has an even number of vertices.
\end{proof}

\section{Cographs} \label{sec:cographs}

As promised, we finish our paper with the following characterization of strongly even-cycle decomposable cographs. 

\begin{COR}\label{cor:cographs}
    Let $G$ be a cograph with no isolated vertices. Then $G$ is strongly even-cycle decomposable if and only if $G$ is $2$-connected Eulerian and $G$ is neither $K_5$ nor $K_5$ with an edge subdivided.
\end{COR}

\begin{proof}
    The forward direction is trivial. Let us prove the converse.
It is well known that if $G$ is a cograph with at least two vertices, then either $G$ or its complement is disconnected~\cite{CLB1981}.
We may assume that $G$ has at least $4$ vertices, as  $K_3$ is trivially strongly even-cycle decomposable.
Since $G$ is connected, there exists a partition $(L,R)$ of $V(G)$ such that all edges are present between $L$ and $R$.  Note that
$\abs{L}$ and $\abs{R}$ cannot both be odd.  We choose $\abs{R}\ge 2$ and $\abs{L}$ maximum. If $\abs{L}\ge 2$, then we are done by Theorem~\ref{thm:join}. Suppose $\abs{L}=1$. Then because $G$ is $2$-connected, $G[R]$ is a connected cograph and so $G[R]$ is the join of two graphs $G_1$ and $G_2$ 
with $\abs{V(G_1)} \leq \abs{V(G_2)}$.  But now $(L \cup V(G_1), V(G_2))$ contradicts the choice of $(L,R)$.  
\end{proof}

%\bibliographystyle{abbrv}
%\bibliography{references-ecd}

\end{document}